\documentclass[11 pt]{article}
\usepackage{verbatim}
\usepackage{amsmath}
\usepackage{latexsym}
\usepackage{amssymb}
\usepackage{amscd}
\usepackage{amsthm}
\usepackage{setspace}
\newtheorem{Thm}{Theorem}[section]
\newtheorem{Pro}[Thm]{Proposition}
\newtheorem{Lem}[Thm]{Lemma}
\newtheorem{Cor}[Thm]{Corollary}
\newtheorem{Def}[Thm]{Definition}
\newtheorem{Conj}[Thm]{Conjecture}

 \makeatletter
\newfont{\footsc}{cmcsc10 at 8truept}
\newfont{\footbf}{cmbx10 at 8truept}
\newfont{\footrm}{cmr10 at 10truept}
\title{Beyond sum-free sets in the natural numbers}

\author
{Sophie Huczynska \footnote{School of Mathematics and Statistics,
University of St Andrews, Fife, KY16 9SS, UK, Email:
sophieh@mcs.st-and.ac.uk}}

\date{
    \begin{center}
    \small AMS classification:
    \small Primary 11B75
    \end{center}
}

\begin{document}
\maketitle

\newcommand{\bc}[2]{{{#1}\choose{#2}}}
\def\mod{\mathop{\mathrm{mod}}}
\def\deg{\mathop{\mathrm{deg}}}


\abstract{For an interval $[1,N] \subseteq \mathbb{N}$,
investigating sets $S \subseteq [1,N]$ such that $|\{(x,y) \in
S^2:x+y \in S\}|=0$, known as sum-free sets, has attracted
considerable attention.  In this paper, we define $r(S)=|\{(x,y) \in
S^2: x+y \in S\}|$ and consider its behaviour as $S$ ranges over the
subsets of $[1,N]$.  We obtain a comprehensive description of the
spectrum of attainable $r$-values for the $s$-sets  of $[1,N]$,
constructive existence results and structural characterizations for
sets attaining extremal and near-extremal values.}

\section{Introduction}
For a finite interval $[1,N] \subseteq \mathbb{N}$, investigating
the nature and number of subsets $S \subseteq [1,N]$ with the
property that $|\{(x,y) \in S^2: x+y \in S \}|=0$ has attracted
considerable attention.  Such sets, called \emph{sum-free}, were
first studied implicitly by Schur in 1916 (\cite{Sch}); in 1988,
interest was revived by Cameron and Erd\H{o}s in \cite{CamErd}, and
their eponymous conjecture regarding the number of such subsets was
later proved by Green (\cite{Gre}) and Sapozhenko (\cite{Sap}).  A
precisely analogous problem can also be considered when $[1,N]$ is
replaced by the integers mod $p$, or indeed by a range of other
abelian (and even non-abelian) groups; see papers such as
\cite{LevSch}, \cite{RheStr} and \cite{Yap}.

In this paper, we remain in the setting of $[1,N]$ (where $N \in
\mathbb{N}$) and consider the quantity $r_N(S):=|\{(x,y) \in S: x+y
\in S \}|$ as $S$ ranges over all subsets of $[1,N]$. When
$r_N(S)=0$, $S$ is of course a sum-free subset of $[1,N]$. Natural
questions which arise are: let $N \in \mathbb{N}$; then for a given
$s \in [1,N]$,  what are the minimum and maximum possible values of
$r_N(S)$ as $S$ runs through all size-$s$ subsets of $[1,N]$?
(Consideration of cardinalities implies it cannot always be $0$.)
What can be said about the structure of sets attaining them? Are all
intermediate values between the maximum and minimum attained by some
$s$-set in $[1,N]$? In a previous paper with Mullen and Yucas
\cite{HucMulYuc}, the author considered the $\mathbb{Z}/ p \mathbb{Z}$ case, and
established best-possible minimum and maximum values.  The precise
nature of the spectrum of values attained between the extremes
remains unknown in this setting, although some partial results and a conjecture are
contained in \cite{HucMulYuc}.

This paper provides a comprehensive description of the situation in
the $[1,N]$ setting.  We describe the spectrum of attainable values,
establish constructive existence results and obtain
characterizations of sets attaining extremal (and some
near-extremal) values. It transpires that in order to answer the
above questions for subsets of a given interval $[1,N]$ ($N \in
\mathbb{N}$), it is helpful to consider the problem from another
angle.  For $s \in \mathbb{N}$, we can ask: what is the range of
cardinalities $|\{(x,y) \in S: x+y \in S \}|$ that can be attained
as $S$ runs through all size-$s$ subsets of $\mathbb{N}$? What is
the smallest $N$ for which all of these values are attained by
subsets of $[1,N]$?  We shall see that the ``tipping point" for the
problem occurs at $N=2s-1$, and that sum-free sets play a crucial
role.

\section{Preliminaries}
Throughout, we use the standard notation $[a,b]$ ($a,b \in
\mathbb{N}$) for the set $\{x \in \mathbb{N}: a \leq x \leq b \}$.

For a finite set $S=\{x_1 < \cdots <x_s\} \subseteq \mathbb{N}$, $S$
may be considered as a subset of any $[1,N]$ with $N \geq x_s$;
since the value of $r_N(S)$ described above remains the same for any
such choice of $N$, we may refer simply to $r(S)$.

\begin{Def}
Let $N \in \mathbb{N}$ and let $S \subseteq [1,N]$. Define
\[ r(S)=|\{(x,y) \in S^2 :x+y \in S \}|. \]
\end{Def}
We shall call $r(S)$ the $r$-value of $S$, and say that $S$ is
$\rho$-closed if it has $r$-value $\rho=r(S)$. Clearly, $0 \leq r(S)
\leq |S|^2$.

\begin{Lem}\label{basic}
Let $N \in \mathbb{N}$.  Let $S=\{x_1 < x_2 < \cdots < x_s \}
\subseteq [1,N]$. Then the following are equivalent definitions of $r(S)$:
\begin{itemize}
\item $r(S)=|\{(x,y) \in S^2:x-y \in S \}|$;
\item $r(S)=\sum_{i=1}^s |S \cap (S+ x_i)|$;
\item $r(S)=\sum_{i=1}^s |S \cap (S- x_i)|$.
\end{itemize}
\end{Lem}

We will use the following notation.
\begin{Def}
Let $N \in \mathbb{N}$ and let $1 \leq s \leq N$.
\begin{itemize}
\item $R(s,N):=\{r(S): S \subseteq [1,N], |S|=s\}$;
\item $f_{s,N}$ is the smallest element in $R(s,N)$;
\item $g_{s,N}$ is the largest element in $R(s,N)$.
\end{itemize}
\end{Def}

When $s=1$, all size-$s$ sets in any interval $[1,N]$ have $r$-value
$0$, i.e. $R(1,N)=1$ for all $N \in \mathbb{N}$.  When $s=2$, the
$r$-value of a size-$s$ set $\{a,b\} \subseteq \mathbb{N}$ equals
$1$ precisely if $b=2a$ and $0$ otherwise; hence $1 \in R(2,N)$ for
any $N \geq 2$ (take $\{1,2\}$), while for $N \geq 3$, we also have
$0 \in R(2,N)$ (take $\{1,3\}$ or $\{2,3\}$).  Henceforth we will
assume $s \geq 3$.

Next we provide, for reference, formulae for certain ``standard
constructions" which will be used throughout the rest of the paper.
We omit proofs in general, as these can easily be supplied by the
interested reader.

\begin{Lem}\label{Int}
Let $S=[i,i+(s-1)] \subseteq [1,N]$.  Then
\[ r(S):= \begin{cases}
\frac{(s-i)(s-i+1)}{2}, & 1 \leq i \leq s-1 \\
0, & s \leq i \leq N-s+1 \\
\end{cases}\]
\end{Lem}

\begin{Lem}\label{AP}
Let $S=\{x,x+a,\ldots,x+(s-1)a\} \subseteq [1,N]$.  Then
\[ r(S):= \begin{cases}
0, & a \nmid x\\
\frac{(s-\gamma)(s-\gamma+1)}{2}, & x=\gamma a \mbox{ and } 1 \leq \gamma \leq s-1\\
0, & x= \gamma a \mbox{ and } s-1< \gamma \leq \frac{N}{a}-(s-1)\\
\end{cases}\]
\end{Lem}

\begin{Lem}\label{IntWithPoint}
Let $S=[1,s] \cup \{x\} \subseteq [1,N]$.  Then
\[ r(S):= \begin{cases}
\frac{s(s-1)}{2}+ (2s+1-x), & s+1 \leq x \leq 2s \\
\frac{s(s-1)}{2}, & x>2s \\
\end{cases}\]
\end{Lem}

\begin{Lem}\label{PunctInt}
Let $S \subseteq \mathbb{N}$ be an interval of size $s$, i.e. $S=[i,i+(s-1)]$, and let $x \in S$.
If $i<s$,
\[ r(S \setminus x)=\frac{(s-i)(s-i+1)}{2}-\mathrm{max}(x-2i+1,0)-\mathrm{max}(2(s-x),0)+\epsilon\]
where $\epsilon:=\begin{cases} 0,x>\frac{i+(s-1)}{2}\\ 1,x \leq
\frac{i+(s-1)}{2} \end{cases}$.

If $i \geq s$, $r(S \setminus x)=r(S)=0$.
\end{Lem}
\begin{proof}
For $i \geq s$, the interval $S=[i,i+(s-1)]$ has $r$-value $0$, and
clearly every subset of a $0$-closed set is $0$-closed.  Assume
$i<s$.  The deleted element $x$ may contribute to $r(S)$  as a
summand (i.e. in pairs of the form $(x,t),(t,x) \in S^2$ such that
$x+t \in S$), or as a sum (i.e. $x=t+u$ for some pair $(t,u) \in
S^2$), or in both roles.  Since all elements are positive integers,
$x$ cannot be a summand in an expression summing to $x$. Each $x \in
S$ occurs as a summand in a sum $x+t$ ($t \in S$) $s-x$ times if $x
\leq s$, and $0$ times otherwise, and similarly for $t+x$ ($t \in
S$).  Doubling this quantity will over-count by one, precisely if
$x+x \in S$, i.e. precisely if $i \leq 2x \leq i+(s-1)$.  So, $x \in
S$ contributes the following to $r(S)$:
\[ \begin{cases}
2(s-x)-1, &  x \leq \frac{i+(s-1)}{2}\\
2(s-x), &  \frac{i+(s-1)}{2} <x \leq s\\
0, &  x>s
\end{cases}\]
Each $x \in S$ occurs as a ordered sum in $S+S$: $0$ times if $x<2i$
and $x-2i+1$ times if $2i \leq x \leq 2i+(s-1)$. Hence, $r(S
\setminus x)$ is equal to $r(S)$ minus the contributions made by $x$
as both summand and sum in $(S+S) \cap S$; the result follows.
\end{proof}

\section{Extremal $r$-values}\label{Extremal}

We begin by considering the extremal $r$-values $g_{s,N}$ and
$f_{s,N}$.

\begin{Thm}\label{gs}
Let $N \in \mathbb{N}$.
\begin{itemize}
\item[(i)] Let $1 \leq s \leq N$.  Then
$$g_{s,N}=\frac{s(s-1)}{2}.$$
\item[(ii)] For $S \subseteq [1,N]$, $r(S)=g_{s,N}$ if and only if $S=\{x_1,\ldots,x_s\}$ is an arithmetic
progression with common difference equal to $x_1$.
\end{itemize}
\end{Thm}
\begin{proof}
Let $S=\{x_1<\ldots<x_s\} \subseteq [1,N]$. By Lemma \ref{basic},
$r(S)=\sum_{i=1}^s |(x_i + S) \cap S|$. For each $i=1,\ldots,s$, all
elements of $x_i+S$ are greater than $x_i$ and hence greater than
$x_j$ with $j<i$.  So $|(x_i + S) \cap S| \leq s-i$, and this bound
is attained precisely if $x_i + S$ contains $S \setminus \{x_1,
x_2,\ldots x_i \} = \{x_{i+1},\ldots,x_s\}$. Thus $r(S)=\sum_{i=1}^s
|(x_i + S) \cap S| \leq  \sum_{i=1}^s (s-i)
=s^2-\frac{s(s+1)}{2}=\frac{s(s-1)}{2}$.  It is clear that, if $S$
is an arithmetic progression with common difference $x_1$, then
$r(S)=\frac{s(s-1)}{2}$.

Now suppose $S$ is an $s$-set in $[1,N]$ with
$r(S)=\frac{s(s-1)}{2}$. Equality must be attained in $|(x_i + S)
\cap S| \leq s-i$ for each $1 \leq i \leq s$, i.e. $x_i + S$
contains $\{x_{i+1}, \ldots, x_{s}\}$ for $1 \leq i \leq s$. When
$i=1$, $S \setminus \{x_1\}=\{x_2<x_3<\ldots<x_s\}$ is contained in
$x_1+S=\{x_1+x_1,x_1+x_2,\ldots,x_1+x_{s-1}, x_1+x_s\}$, where
$x_1+x_s \not\in S$. Hence $x_1+x_1=x_2, x_1+x_2=x_3, \ldots,
x_1+x_{s-1}=x_s$, i.e. $x_2=2x_1, x_3=3x_1,\ldots,x_s=sx_i$, as
required.
\end{proof}

Hence the maximum possible value of a size-$s$ set in $[1,N]$ ($s
\leq N$) does not depend on $N$.

\begin{Cor}
Let $N \in \mathbb{N}$.
\begin{itemize}
\item If $s>\frac{N}{2}$, then $S=[1,s]$ is the unique $s$-set in $[1,N]$ with
$r(S)=g_{s,N}$.
\item If $ks \leq N < (k+1)s$, there are $k$ $s$-sets $S$ in $[1,N]$ such that $r(S)=g_{s,N}$,
given by $S=\{x_1,2x_1,\ldots,sx_1\}$ with $1 \leq x_1 \leq k$.
\end{itemize}
\end{Cor}

We now ask: what is the smallest possible $r$-value for a set of
size $s$?  It is not always possible to obtain a minimum $r$-value
of $0$. For $S=\{x_1< \ldots < x_s\}$, denote by $(S-S)^+$ the set
$(S-S) \cap [1,N]$; clearly $(S-S)^+$ has at least $s-1$ distinct
values. Then $r(S)=0$ if and only if $(S-S)^+ \cap S = \emptyset$.
If $S \subseteq [1,N]$ with $s=|S|
> \lceil \frac{N}{2} \rceil$, the number of distinct values in $(S-S)^+$ is at least $\lceil
\frac{N}{2} \rceil$, meaning $(S-S)^+$ cannot lie entirely within
$[1,N] \setminus S$ and hence must have non-empty intersection with
$S$.

In fact, the following result describes the situation precisely.

\begin{Thm}\label{fs}
Let $N \in \mathbb{N}$ and $1 \leq s \leq N$. Then
\[
f_{s,N}=\begin{cases}
0, & s \leq \frac{N+1}{2} \\
\frac{(2s-N)(2s-N-1)}{2}, & s>\frac{N+1}{2} \\
\end{cases}
\]
\end{Thm}
\begin{proof}
Let $S=\{x_1<\ldots<x_s\}$ and let $T:=[1,N] \setminus S$ with
$|T|=t$. From Lemma \ref{basic}, $r(S)=\sum_{i=1}^s |(S-x_i) \cap
S|$. We obtain a lower bound for $r(S)$ by establishing an upper
bound for $q(S,T):=|\{(x,y) \in S^2:x-y \in T\}|$, i.e. for
$\sum_{i=1}^s |(S-x_i) \cap T|$. The set $(S-x_1) \cap [1,N]$ has
$s-1$ distinct elements, of which all may be from $T$ if $s-1 \leq
t$, and at most $t$ may be from $T$ if $s-1>t$; so $|(S-x_1) \cap T|
\leq \mathrm{min}(t,s-1)$. In general $|(S-x_i) \cap T| \leq
\mathrm{min}(t,s-i)$ ($i=1,\ldots,s-1$), and so $q(S,T) \leq
\sum_{i=1}^s \mathrm{min}(t,s-i)$.  Hence if $t \geq s-1$, then
$q(S,T) \leq \sum_{i=1}^s (s-i)$, while if $t=s-j$ for some
$j=1,\ldots,s-1$ then $q(S,T) \leq (j-1)t + \sum_{i=j}^s (s-i)$.

Since $r(S)+q(S,T)=|\{(x,y) \in S^2: (x-y) \in
[1,N]\}|=\frac{s(s-1)}{2}$, clearly $r(S) \geq \frac{s(s-1)}{2}-
\sum_{i=1}^s \mathrm{min}(t,s-i)$.  Hence if $t \geq s-1$, then
$r(S) \geq \frac{s(s-1)}{2} - s(s-1) + \sum_{i=1}^s i = 0$; while if
$t=s-j$ for some $j=1,\ldots,s-1$ then $r(S) \geq
(j-1)((s-t)-\frac{j}{2})$. Using $j=s-t$ and $t=N-s$ yields the
stated lower bound.

To see that this is best possible, consider the $s$-set $S=[N-s+1,N]
\subseteq [1,N]$ and apply Lemma \ref{Int}.
\end{proof}

The above proof can be exploited to characterize the structure in
the case when $s \geq \frac{N+1}{2}$.

\begin{Thm}\label{fstructure}
Suppose $S \subseteq [1,N]$ with $s=|S| \geq \frac{N+1}{2}$ and
$r(S)=f_{s,N}$.  Then:
\begin{itemize}
\item if $s> \frac{N+1}{2}$ or $N$ is even, then $S=[N-s+1, N]$;
\item if $N$ is odd and $s=\frac{N+1}{2}$, then $S=[N-s+1,N]$ or $S=\{1,3,5,\ldots,2(N-s)+1
\}$.
\end{itemize}
\end{Thm}
\begin{proof} Let $S= \{x_1< \cdots <x_s \}$ and $T=[1,N] \setminus S=\{y_1< \cdots
<y_t\}$. Assume $s \geq \frac{N+1}{2}$, so $t \leq s-1$. From the
proof of Theorem \ref{fs}, since $r(S)+q(S,T)=\frac{s(s-1)}{2}$, the
minimum possible value $f_{s,N}$ of $r(S)$ is attained precisely if
the maximum possible value of $q(S,T)$ is attained. This occurs if
and only if $|(S-x_i) \cap T|$ has maximum possible value
$\mathrm{min}(t,s-i)$ for each $i=1,\ldots,s-1$. Now, $t=s-j$ for
some $1 \leq j \leq s-1$. Hence $|(S-x_i) \cap T|=t$ for
$i=1,\ldots,j$ and $|(S-x_i) \cap T|=s-i$ for $i=j+1,\ldots,s-1$.
When $i=j$, we have $(S-x_i) \cap T=T$, i.e. $\{x_{j+1}-x_j< \cdots
< x_s-x_j \}=\{y_1 < \cdots < y_t \}$.  This forces
$y_1=x_{j+1}-x_j$ and in general $y_k=x_{j+k}-x_j$ for
$k=1,\ldots,t$, i.e. $x_{j+k}=x_j+y_k$ for $k=1,\ldots,t$. Next,
consider $(S-x_{j+1}) \cap T$, which is a proper subset of $T$. Its
$t-1$ elements are $\{x_{j+2}-x_{j+1}< x_{j+3}-x_{j+1}< \cdots <
x_s-x_{j+1} \}$, i.e. $\{y_2-y_1 < y_3-y_1 < \cdots < y_s-y_1 \}$
using the expressions obtained above. Now, $y_2-y_1<y_2$ and is an
element of $T$, so $y_2-y_1=y_1$. Similarly, $y_3-y_1 \in T$ and
$y_1<y_3-y_1<y_3$, so $y_3-y_1=y_2$. Thus $y_2=2y_1$, $y_3=3y_1$ and
in general $y_k=ky_1$ ($1 \leq k \leq t$).  Thus $T$ is the
arithmetic progression $\{ y_1, 2y_1,\ldots t y_1 \}$, i.e. the
arithmetic progression $\{a,2a,\ldots,(N-s)a\}$ where
$a=x_{2s-N+1}-x_{2s-N}$.

We now determine the precise nature of $S$ and $T$.  From the case
$i=j$, we see $\{x_{j+1}< \cdots < x_s \}= x_j + T$, i.e. the last
$t$ terms of $S$ form an arithmetic progression with constant term
$y_1$. Since $S=[1,N] \setminus T$, we have the following
possibilities for $S$ and $T$.
\begin{itemize}
\item $y_1=1$, i.e. $T$ is an interval. In this case $T=[1,t]$ and $S=[t+1,N]$.
\item $y_1=2$.  In this case, $T=\{2,4,\ldots,2t\}$ and
$S=\{1,3,\ldots, 2t \pm 1 \}$ ($S$ can possess no further elements
otherwise the last $t$ elements of $T$ would not be an arithmetic
progression with common difference $2$). Since by assumption $s \geq
t+1$, $x_s=2t-1$ is not possible, so $S$ must be $S=\{1,3,\ldots,
2t+1 \}$, and so $s=t+1$, i.e. $2s=N+1$.
\item $y_1 \geq 3$.  This is not possible since if $T=\{k,2k,\ldots,kt\}$ with $k \geq
3$ then $S$ cannot have its last $t$ terms forming an arithmetic
progression with common difference $3$.
\end{itemize}
\end{proof}

In the case when $s \leq \frac{N}{2}$, the techniques of the above
proof cannot be used to characterize the structure of an
$f_{s,N}$-closed set with $f_{s,N}=0$.  In fact, the difficulty of
determining the structure of all sum-free sets is well-known.  We
observe that the case $t=s$ with $N=2s$ can be shown to have the
same two possibilities as the $t=s-1$ case, namely an interval or
arithmetic progression with common difference $2$. Section
\ref{D_S_stuff} establishes results on the structure of $r$-closed
sets with small $r$-values, including this case.

\section{Sets corresponding to small $r$-values}\label{D_S_stuff}

In this section, techniques are developed which allow us to describe
the structure of sets with $r$-values equal or close to $f_{s,N}$.

\subsection{Set structure when $(S-S) \cap [1,N]$ is small}

Let $S=\{x_1<x_2< \cdots <x_s\}$ be a subset of $[1,N]$ of size $s$,
$1 \leq s \leq N$. As before, denote by $(S-S)^+$ the set $(S-S)
\cap [1,N]$.  It is clear that $|(S-S)^+| \geq s-1$, since
$\{x_s-x_{s-1}<x_s-x_{s-2}< \cdots x_s-x_1 \} \subseteq (S-S)^+$.

Define the $i$th difference vector $D_S(i)$ of $S$ as follows ($1
\leq i \leq s-1$): $D_S(1)=(x_2-x_1,x_3-x_2,\ldots,x_s-x_{s-1})$,
$D_S(2)=(x_3-x_1,x_4-x_2,\ldots,x_s-x_{s-2})$ and in general
$D_S(i)=(x_i-x_1,\ldots, x_s-x_{s-i})$.  It is clear that, once
$D_S(1)$ is specified, this completely defines the other $D_S(i)$
for $2 \leq i \leq s-1$, since the $j$th entry of $D_S(i)$ is the
sum of the $j$th to $(j+i)$th consecutive entries in $D_S(1)$. Hence
a subset $S=\{x_1< \cdots <x_s\}$ of $[1,N]$ is uniquely defined by
specifying its 1st difference vector $D_S(1)$ together with value
$x_1$.

\begin{Lem}\label{intermediate}
Let $S=\{x_1<x_2< \cdots <x_s\}$ be a subset of $[1,N]$.  Let $0
\leq k \leq x_s-x_1$. Then
\begin{itemize}
\item[a)] $|(S-S)^+|\leq (s-1)+k$ $\Leftrightarrow$ for all $j$ with $1
\leq j \leq s-1$, $\cup_{i=1}^j D_S(i)$ contains at most $j+k$
distinct values.
\item[(b)] If (a) holds for a set $S$, then if $\cup_{i=1}^j D_S(i)$ contains precisely $j+k$
distinct values for some $1 \leq j <s-1$, $D_S(l) \setminus
D_S(l-1)$ must contain precisely one value for all $l>j$.
\end{itemize}
\end{Lem}
\begin{proof}
For (a), $|(S-S)^+|$ is at least $s-1$ and at most $x_s-x_1$. The
reverse implication is easily seen to hold by taking $i=s-1$.  For
the forward implication, let $|(S-S)^+|\leq (s-1)+k$ but suppose
that, for some $1 \leq j < s-1$, $\cup_{i=1}^j D_S(i)$ contains at
least $j+k+1$ distinct values. Choose $j$ to be the smallest such.
But then each of the $s-1-j$ sets $D_S(j+1),\ldots,D_S(s-1)$
contributes at least one element which did not occur in the previous
sets (namely $x_s-x_{s-j-1}, \ldots, x_s-x_1$), and so $|(S-S)^+|
\geq (j+k-1)+(s-j+1)=s+k$ which is impossible. For part (b), part
(a) implies that $D_S(j+1) \setminus D_S(j)$ contains at most one
value, while clearly it contains at least one value namely
$x_s-x_{s-j-1}$, so it contains precisely one value. Hence
$\cup_{i=1}^{j+1} D_S(i)$ contains exactly $j+k+1$ distinct values;
repeating this argument proves the claim for all $l>j$.
\end{proof}

We have the following immediate corollary:
\begin{Cor}\label{SisAP}
Any $s$-set $S \subseteq [1,N]$ with $|(S-S)^+|=s-1$ is an
arithmetic progression.
\end{Cor}

\begin{Pro}\label{SisAlmostAP}
Let $S \subseteq [1,N]$ with $|S|=s>3$ and $r(S)=0$. Suppose
$|(S-S)^+|=s$.  For $s>4$, $S$ has one of the following forms:
\begin{itemize}
\item[(i)] $\{x,x+a,\ldots,x+sa\} \setminus \{x+ia\}$ where $1 \leq
i \leq s-1$; or
\item[(ii)] $\{x,x+a,\ldots,x+(s+1)a\} \setminus \{x+a,x+sa\}$
\end{itemize}
where $x,a \in [1,N]$.  For $s=4$, $S$ is either of type (i), (ii)
or of the form
\begin{itemize}
\item[(iii)] $S=\{x,x+a,x+a+b,x+2a+b \}$, where $x,a,b \in [1,N]$.
\end{itemize}
\end{Pro}
\begin{proof}
Let $s \geq 5$. We prove the equivalent claim that $D_S(1)$ has the
form $(a^i, 2a, a^j)$ or $(2a, a^k, 2a)$ for some $i,j,k \geq 0$
(where the notation $a^i$ denotes $i$ consecutive entries, each with
value $a$). By Lemma \ref{intermediate} (a), $\cup_{i=1}^j D_S(i)$
contains at most $j+1$ distinct values for $1 \leq j \leq s-1$. In
fact, $D_S(1)$ contains exactly two values, since a single-valued
set would correspond to an arithmetic progression and hence
$|(S-S)^+|=s-1$. Thus by Lemma \ref{intermediate} (b), precisely one
new value occurs in moving from $D_S(i)$ to $D_S(i+1)$ ($1 \leq i
<s-1$).

Consider the vector $D_S(1)$ as corresponding to a word (of length
$s-1$) in two symbols $\{a,b\}$, ($a,b, \in (S-S)^+$). We now ask:
which words form valid vectors? Consideration of $D_S(2)$ shows that
\begin{itemize}
\item subwords $baa$, $aab$ are valid only if $b=2a$;
\item subwords $abb$, $bba$ are valid only if $a=2b$;
\item subwords $aabb$, $bbaa$ are invalid;
\end{itemize}
while consideration of $D_S(3)$ shows
\begin{itemize}
\item subwords $abab$, $baba$ are invalid.
\end{itemize}
Hence all valid words must consist of symbols $a$ and $b=2a$,
subject to the above subword restrictions which together imply that
any two occurrences of the symbol $b=2a$ must be separated by a
sequence of length at least 2 of consecutive occurrences of symbol
$a$. Furthermore, since the entries of $D_S(i)$ correspond to the
sums of the length-$i$ subwords of $D_S(1)$, Lemma
\ref{intermediate} (b) implies that the entries of the length-$i$
subwords sum to at most $(i+1)a$ for $1 \leq i <s-1$. Hence there
can be no $i$ consecutive entries in the vector $D_S(1)$ containing
more than one occurrence of $2a$, unless $i=2a$. So the two possible
forms for $D_S(1)$ are $[a,\ldots,a,2a,a,\ldots,a]$ or
$[2a,a,\ldots,a,2a]$.  For the $s=4$ case, $D_S(1)=aba$ is valid for
any choice of $a,b, \in (S-S)^+$, while $D_S(1)$ with two identical
consecutive entries is restricted as above.
\end{proof}

Note that, for $s=2,3$, this approach gives no restriction on the
structure of $D_S(1)$.

One immediate application of these results is to establish the
following facts about $0$-closed sets.

\begin{Pro}\label{0closedOdd}
Let $N \in \mathbb{N}$ be odd, and let $S$ be a $0$-closed set of
maximum size $\lceil \frac{N}{2} \rceil$ in $[1,N]$.  Then $S$ is an
arithmetic progression with common difference $1$ or $2$.
\end{Pro}
\begin{proof}
Let $N=2K-1$.  Since $|S|=\lceil \frac{N}{2} \rceil=K$, and $S$ is
$0$-closed, $|(S-S)^+| \leq N-|S|=K-1=|S|-1$, i.e. $|(S-S)^+|=s-1$.
By Corollary \ref{SisAP}, $S$ is an arithmetic progression, which
from size considerations must have common difference $1$ (in which
case $S=[K,2K-1]=[\frac{N+1}{2},N]$) or $2$ (in which case it is the
set of
odd numbers in $[1,N]$, i.e. $S=\{1,3,\ldots,2K-1\}$).\\
\end{proof}

\begin{Pro}\label{0closedEven}
Let $N \in \mathbb{N}$ be even, $N>8$, and let $S$ be a $0$-closed
set of maximum size $\frac{N}{2}$ in $[1,N]$.  Then $S$ is an
arithmetic progression with common difference $1$ or $2$.
\end{Pro}
\begin{proof}
Let $N=2K$.  Since $|S|=K$, and $S$ is $0$-closed, $|(S-S)^+| \leq
N-K=K$, i.e. possible sizes are $|(S-S)^+|=K-1$ or $K$.  In the
first case, $S$ is an arithmetic progression with common difference
$1$ or $2$.  There are two possible forms in the interval case,
$S=[K,2K-1]$ and $S=[K+1,2K]$, and a single possibility in the other
case, $S=\{1,3,\ldots,K-1\}$. Otherwise, $|(S-S)^+|=K$. From the
proposition above, $S$ must be an arithmetic progression of length
$K+1$ or $K+2$ in $[1,2K]$ with $1$ or $2$ non-extremal points
deleted; this is possible only if the arithmetic progression is an
interval. But then $(S-S)^+$ contains $\{ 1,2,\ldots,K \}$ which
must not intersect with $S$; however $S=[x,x+K+1] \setminus (x+i)$
(some $1\leq i \leq x+K$) or $S=[x,x+K+2] \setminus \{x+1,x+K+1\}$,
so this case is impossible as $S$ cannot lie within $[1,2K]$.
\end{proof}

If $N=8$, then either $S$ is an arithmetic progression with common
difference $1$ or $2$, or $S=\{x,x+a,x+a+b,x+2a+b\}$ for some $x,a,b
\in [1,N]$.  In fact, the four $0$-closed sets of size $4$ are:
$\{4,5,6,7\}$, $\{5,6,7,8\}$, $\{1,3,5,7\}$ and $\{2,3,7,8\}$.  Note
that for $N=4$, every $2$-set except $\{1,2\}$ and $\{2,4\}$ is
$0$-closed.

\subsection{Sets with small $r$-values}

Here we prove a structural result and, using the same technique, a
non-existence result which will form the base case of an inductive
argument in the next section. Throughout this section, we assume
$s>3$.

\begin{Thm}
Let $N=2s-1$ (where $s>3$). Suppose $S \subseteq [1,N]$ with $|S|=s$
and $r(S)=1$.  Then

\begin{itemize}
\item $S=[\frac{N-1}{2},N-1]$; or
\item $S=[\frac{N-1}{2},N] \setminus \{\frac{N+1}{2}\}$.
\end{itemize}
\end{Thm}
\begin{proof}
By Theorem \ref{fs}, $f_{s,N}=0$.  Suppose $S$ has $|S|=s$ and
$r(S)=1$. Then there exists a single element $x \in S \cap (S+S)$
which has a unique expression $x=a+a$ as a sum in $S+S$ ($a \in S$
and $a \leq \frac{N-1}{2}$). Consider $V:=S \setminus \{x\}$. $V$ is
$0$-closed, of size $v:=\lceil \frac{N}{2} \rceil-1=\frac{N-1}{2}$.
Since $V \cap (V-V)= \emptyset$, $|(V-V)^+| \leq N-|V|$.  If $V$ and
$(V-V)^+$ partition $[1,N]$, then since $x \not\in V$, $x \in
(V-V)^+$, i.e. there exist $v,w \in V$ such that $x=v-w$.  But then
$v=(w+x) \in V \subseteq S$ and $w+x \in V+S \subseteq S+S$, so $v
\in S \cap (S+S)$, impossible since $v \neq x$.  Hence $V$ and
$(V-V)^+$ do not partition $[1,N]$, and $|(V-V)^+| \leq N-|V|-1=v$.
So $v-1 \leq |(V-V)^+| \leq v$.

\noindent {\bf Case: $|(V-V)^+|=v-1$.}  Here $V$ is an arithmetic
progression, which by size considerations has common difference $1$
or $2$.  If the difference is $1$, $V$ is an interval $[i,i+(v-1)]$;
by Lemma \ref{Int}, $v \leq i \leq N-v+1$, i.e. $i \geq
\frac{N-1}{2}$. Since $a \in V$ and $a\leq \frac{N-1}{2}$, we must
have $a=i=\frac{N-1}{2}=v$.  Then $V=[\frac{N-1}{2},N-2]$ and $S=V
\cup \{x\}$ where $x=2a=N-1$.  Otherwise, $V$ is either
$\{1,3,\ldots,N-2\}$ or $\{3,5,\ldots,N\}$.  Since $a \in V$, we
must have $x=2a$ greater than the largest element of $V$, so the
only possibility is $\{1,3,\ldots,N-2\}$ with $x=N-1$. But this
would give $r(S)>1$, and so cannot occur.

\noindent{ \bf Case: $|(V-V)^+|=v$.} We first assume that $N>9$,
i.e. $v \geq 5$.  Here $V$ is either a $(v+1)$-term arithmetic
progression with one non-extremal element deleted, or a $(v+2)$-term
arithmetic progression with its second and second-last elements
deleted. In the former case, we have that
$D_V(1)=\{\alpha,\ldots,\alpha,2\alpha,\alpha,\ldots,\alpha\}$,
where $\alpha$ is the common difference in the arithmetic
progression, i.e. $1$ or $2$ here.  Then
$(V-V)^+=\{\alpha,2\alpha,\ldots,v \alpha\}$.  If $\alpha=2$, then
$V=\{1,3,\ldots,N\}\setminus \{\beta\}$ for some odd $1<\beta<N$ and
$(V-V)^+=\{2,4,\ldots,N-1\}$ contains all even numbers in $[1,N]$.
But $x=2a \in S \subseteq [1,N]$, and so $x \in (V-V)^+$.  But this
is impossible, so this case cannot occur.  If $\alpha=1$, then
$V=[i,i+\lceil\frac{N}{2}\rceil+1]\setminus{\beta}$ for some $i$ and
some $i<\beta<i+\lceil\frac{N}{2}+1\rceil$.  Here
$(V-V)^+=\{1,2,\ldots,v=\frac{N-1}{2}\}$ and $a \leq \frac{N-1}{2}$,
so we must have $a \in (V-V)^+$, i.e. there exist $v,w \in V$ such
that $a=v-w$, i.e. $v=w+a$.  But then $v \in S \cap (S+S)$ yet $v
\neq x$, impossible.

The only remaining case is that $V$ is a $(v+2)$-term arithmetic
progression with second and second-last elements deleted.  Here,
$D_V(1)=\{2\alpha,\alpha,\ldots,\alpha,2\alpha\}$ and
$(V-V)^+=\{\alpha,2\alpha,\ldots,(v-1)\alpha,(v+1)\alpha\}$, where
$\alpha$ is the common difference, which here must be $\alpha=1$ by
size considerations.  So $V$ has form
$[i,i+\lceil\frac{N}{2}\rceil]\setminus\{i+1,i+\lceil\frac{N}{2}\rceil-1\}$
for some $i$ and $(V-V)^+=[1,v+1]\setminus\{v\}$.  Now, $a \in V$
and $a \leq \frac{N-1}{2}=v$, so $i\leq v$.  But (as above) $a
\not\in (V-V)^+$; hence $a=v$ and $x=2a=N-1$.  Since no element of
$V$ is in $(V-V)$, we must have that $i=v$ and $\{v+1,N-1\}$ are the
deleted elements.  So in this case $V=[\frac{N-1}{2},N] \setminus
\{\frac{N+1}{2},N-1\}$ and $S=[\frac{N-1}{2},N-1] \setminus
\{\frac{N+1}{2}\}$. Finally, the case when $v=4$, i.e. $N=9$, can be
established by direct verification, e.g. using GAP \cite{Gap}.
\end{proof}

We remark that an analogous proof technique can be applied to
establish the structure of other $s$-sets in $[1,N]$ with $s$ close
to $\frac{N}{2}$ and $r(S)$ close to $0$. We now use a similar
approach to prove a non-existence result.

\begin{Thm}\label{risnot2}
Let $N=2s-2$ (where $s>3$).  Then there exists no $S \subseteq
[1,N]$ of size $s$ with the property that $r(S)=2$.
\end{Thm}
\begin{proof}
By Theorem \ref{fs}, $f_{s,N}=1$.  We first suppose that $N>8$. With
a view to obtaining a contradiction, we suppose that there exists
$S$ of size $s$ with $r(S)=f_{s,N}+1=2$.

There are two possibilities for $S$:
\begin{itemize}
\item[(a)] There exists precisely one $x \in S \cap (S+S)$, which has
precisely two representations as a(n ordered) sum in $S+S$, namely
$x=a+b=b+a$ for $a \neq b \in S$;
\item[(b)] There exist precisely two elements $x\neq y \in S \cap (S+S)$,
and each has a single representation as a(n ordered) sum in $S+S$:
$x=a+a$, $y=b+b$ for some $a \neq b \in S$.
\end{itemize}
\noindent{\bf Case (a)}  Deleting $x$ from $S$ yields a $0$-closed
set of maximum size $\frac{N}{2}$, which must be either an interval
(i.e. $[\frac{N}{2},N-1]$ or $[\frac{N}{2}+1,N]$), or
$\{1,3,5,\ldots N-1 \}$. First suppose $S=\{x\} \cup [i,i+(s-2)]$
($i=s-1$ or $s$).  The case $x>i+(s-2)$ can occur only if the
interval is $[\frac{N}{2},N-1]$ and $x=N$; but then $r(S)=1$.  So
$x<i$, i.e. $1 \leq x \leq s-1$; hence $2x \in [i,i+(s-2)]$ and $x
\in D_S(1)$, but then $r(S)=1+2k$ for $k \geq 1$.  So the only
possible case is that $S=\{x\} \cup \{1,3,\ldots,N-1\}$; but then
$x=2k$ for some $1 \leq k \leq N$ and $r(S)=2s-k \geq s
>2$.  So this case is not possible.

\noindent{\bf Case (b)} Recall that $x,y$ are the unique elements of
$S \cap (S+S)$ and $x=a+a$, $y=b+b$ are their unique expressions as
sums in $S+S$; let $x<y$ and hence $a<b$. We have $a<b<y$ and
$a<x<y$, so it is possible to have $b=x$ but no other equalities can
hold between these four elements.  Clearly, since $2a,2b \leq N$,
$a,b \leq \frac{N}{2}$.

\noindent{\bf Subcase: $b=x$} Suppose $b=x$, i.e. $a+a=x=b$ and
$b+b=y=4a$. Then deleting $x$ from $S$ yields a $0$-closed set of
maximum size $\frac{N}{2}$, which must be an interval or
$\{1,3,5,\ldots N-1 \}$.  As above in Case (a), $S=\{x\} \cup
\{1,3,5,\ldots N-1 \}$ is impossible. So $S \setminus \{x \}$ is an
interval $[i,i+(s-2)]$.  The case $x>i+(s-2)$ can occur only if the
interval is $[\frac{N}{2},N-1]$ and $x=N$; but $2x \in S \subseteq
[1,N]$ so this is impossible. Thus $x<i$, but this is also
impossible since $i\leq a <b$.  So this case is not possible.

\noindent{\bf Subcase: $b \neq x$} Here $b \neq x$, the $r$-value of
$S$ with any one of $\{a,b,x,y\}$ deleted is $1$, and $r$-value of
$S$ with $\{a,b\}$, $\{x,y\}$, $\{a,y\}$ or $\{b,x\}$ deleted is
$0$. Consider $U:=S \setminus \{b,x\}$ ($x=2a$).  The set $U$ is a
$0$-closed set of size $u=\frac{N}{2} - 1$, which contains
$\{a,2b\}$. Since $a<b\leq \frac{N}{2}$, $a <s$ and so $U$ is not
contained in an interval of the form $[s,2s-1]$ nor $[s+1,2s]$.
Since $U$ contains $2b$, $U$ is not contained in the set of odd
numbers $\{1,3,5\ldots, N-1 \}$. Hence $U$ is not contained in a
$0$-closed set of size $\frac{N}{2}$, hence is a maximal $0$-closed
set of size $\frac{N}{2} -1$.

Now, since $U$ is $0$-closed, $U \cap (U-U)=\emptyset$ and so $u-1
\leq |(U-U)^+| \leq N-u$, i.e. $\frac{N}{2} -2 \leq
|(U-U)^+| \leq \frac{N}{2}+1$. However, if $U$ and $(U-U)^+$ partition $[1,N]$, then since $b
\not\in U$, we must have $b \in (U-U)^+$, i.e. there exist $v,w \in
U$ such that $b=v-w$ and hence $v=(w+b) \in (S+S) \cap S$.  But the
only two elements in $(S+S) \cap S$ are $x=2a$ and $y=2b$ with
unique expressions in $S+S$ as $a+a$ and $b+b$; however $v \neq 2a$
since $v \in U$ and $v \neq 2b$ because $w \neq b$ since $b \in U$.
So $b \in [1,N] \setminus (U \cup (U-U)^+)$.  Similarly, if $U$ and
$(U-U)^+$ partition $[1,N]$, then $2a \in (U-U)^+$, i.e. there exist
$v,w \in U$ such that $2a=v-w$ and hence $v=(w+2a) \in (S+S) \cap
S$. But then $v \neq 2a$ since $v>2a$, and $v \neq 2b$ since $w,2a
\neq b$. So $\{b,2a \} \in [1,N] \setminus (U \cup (U-U)^+)$ ($b
\neq 2a$). Hence $|(U-U)^+| \leq N-u-2=u$.

It now remains to show that the options $|(U-U)^+|=u-1$ and $u$ lead to a contradiction.

\noindent $\bullet$ {$|(U-U)^+|=u-1$:} By Corollary \ref{SisAP}, $U$
is an interval. Any $0$-closed interval of size $u=\frac{N}{2} -1$
must be $[i,i+ \frac{N}{2} -2]$ where $i \geq \frac{N}{2} -1$.  But
$a \in U$ and $a< \frac{N}{2} -1$, hence this case is impossible.

\noindent $\bullet$ {$|(U-U)^+|=u$}: Assume first that $u \geq 5$.
By Proposition \ref{SisAlmostAP}, $U$ is either a $(u+1)$-term
arithmetic progression with one (non-extremal) term deleted, or a
$(u+2)$-term arithmetic progression with the second and second-last
terms deleted.  By size considerations, such an arithmetic
progression must either be an interval or, when $U$ has $u+1$ terms,
$\{1,3,5,\ldots,N-1\}$. But since $2b \in U$, the latter is not
possible, so $U$ is an interval with one or two elements deleted.

First suppose $U$ is an interval of length $u+1=\frac{N}{2}$ with one element deleted.  Since $U$ is maximal $0$-closed,
the interval of length $\frac{N}{2}$ cannot itself be
$0$-closed, and so must be $[i,i+\frac{N}{2} -1]$ for
$i<\frac{N}{2}$.  Thus $U=S \setminus
\{b,2a\}=[i,i+\frac{N}{2} -1] \setminus \{ \alpha\}$
for some $i<\alpha<i+\frac{N}{2} -1$. We ask: where do
$b$ and $2a$ lie?  We cannot have $2a<i$ nor $b>i+\frac{N}{2} -1$ since $a,2b \in U$; similarly we cannot have $b<i$ nor
$2a>i+\frac{N}{2} -1$. Since $b \neq 2a$, these two
points cannot both equal the deleted point $\alpha$ within
$[i,i+\frac{N}{2} -1]$.  So this case cannot occur.

Now suppose $U$ is an interval of length $u+2=\frac{N}{2} + 1$ whose
second and second-last elements have been deleted, i.e.$[i,i+
\frac{N}{2}] \setminus \{\alpha,\beta\}$ where $\alpha=i+1$,
$\beta=i+ \frac{N}{2}-1$.  Arguing as above, since $a,2b \in U$, we
cannot have $b,2a<i$ nor $b,2a>i+\frac{N}{2}$, so we must have
$\{a,2b\}=\{\alpha,\beta\}$.  Hence $S$ must be an interval of
length $\frac{N}{2} + 1$ in $[1,N]$.  But an interval cannot have
$r$-value $2$ (by Lemma \ref{Int}, an interval has an $r$-value of
the form $\frac{k(k-1)}{2}$ for some $k$), so this is impossible.

Hence there exists no subset $S$ of $[1,N]$ with $|S|=s=
\frac{N}{2} + 1$ which has $r(S)=f_s+1=2$, for $N>8$.  Direct
checking (e.g. computationally using GAP) establishes the result for
$N=6$ and $N=8$.
\end{proof}

We immediately have the following consequence.
\begin{Cor}
Let $5 \leq N \in \mathbb{N}$.  Then there exists no $s$-set $S
\subseteq [1,N]$ with $N<2s-1$ and $r(S)=2$.
\end{Cor}
\begin{proof}
If $N=2s-2$, then $s>3$ and $f_{s,N}=1$, hence Theorem \ref{risnot2}
applies.  Otherwise, $N \leq 2s-3$ and $f_{s,N} \geq 3$ by Theorem
\ref{fs}.
\end{proof}

Note that for $N=4$, there are two sets of size $\frac{N}{2}+1=3$ with $r$-value
$2$, namely $\{1,2,4\}$ and $\{1,3,4\}$.

\section{Establishing the spectrum of $r$-values}

At the outset, we posed the following question: for $N \in
\mathbb{N}$ and  $1 \leq s \leq N$, does there exist a size-$s$ set
$S$ with $r(S)=v$ for each $v \in [f_s,g_s]$?  In the previous
section, it was shown that the answer cannot be in the affirmative
for every $N$ and $s$.  To address this question in a general
setting, we will consider the problem from a different angle; first
specify $s \in \mathbb{N}$, and make the choice of interval a
secondary consideration.

Let $s \in \mathbb{N}$.  What are the $r$-values that a size-$s$ set
in $\mathbb{N}$ can attain?  In other words, what is $R(s):=\{
r(S):S \subseteq \mathbb{N}, |S|=s\}$?  The maximum possible value
in $R(s)$ is $\frac{s(s-1)}{2}$, and this is attainable in any
interval $[1,N]$ with $N \geq s$ (take the set $[1,s]$; there are
other possible $s$-sets for sufficiently large $N$).  At the other
extreme, Section \ref{Extremal} showed that the minimum possible
value in $R(s)$ is $0$, but that this is attainable in an interval
$[1,N]$ only if $N \geq 2s-1$.  For $s \leq N<2s-1$, the minimum
$r$-value for a size-$s$ set in $[1,N]$ is given by
$f_{s,N}=\frac{(2s-N)(2s-N-1)}{2}$.

In the remainder of the paper, we will prove the following result
describing the spectrum of values.  Let $s (>2) \in \mathbb{N}$.
\begin{itemize}
\item[(i)] For $N=s$ and $N=s+1$, $R(s,N)=[f_{s,N},g_{s,N}]=[\frac{(2s-N)(2s-N-1)}{2},\frac{s(s-1)}{2}]$ with $f_{s,N}>0$;

\item[(ii)] for  $s+2 \leq N \leq 2s-2$, $R(s,N)=[f_{s,N},g_{s,N}] \setminus E(s,N)$ where
$f_{s,N}>0$, and the set $E(s,N)$ of exceptions is non-empty and is
contained within the set $\{f_{s,N}+(2i-1):1 \leq i \leq
\mathrm{min}(s-\lceil \frac{N}{2} \rceil, \lfloor \frac{N-s}{2}
\rfloor) \}$;

\item[(iii)] for $N \geq 2s-1$,
$R(s,N)=[f_{s,N},g_{s,N}]=[0,\frac{s(s-1)}{2}]$.
\end{itemize}

\subsection{The non-exceptional range}

Throughout, let $N,s \in \mathbb{N}$.  For $s=3$, it is easy to
check that $R(3,3)=\{3\}$, $R(3,4)=[1,3]$ and $R(3,N)=[0,3]$ for $N
\geq 5$; hence we may assume that $s>3$. For $N=s$, the desired
result holds trivially, since the only $s$-set is the whole interval
$[1,s]$. The next proposition shows that the result also holds for
$N=s+1$.

\begin{Pro}\label{N=s+1}
For $N\geq s+1$, $R(s,N)$ contains the interval
$[\frac{(s-1)(s-2)}{2},\frac{s(s-1)}{2}]$.
\end{Pro}
\begin{proof}
All $r$-values are obtained as $r([1,s+1] \setminus x)$ for $x \in
[1,s+1]$; apply Lemma \ref{PunctInt} to the size $t=s+1$ interval
$T:=[i,i+(t-1)]$ with $i=1$.
\end{proof}

We now establish results leading to a proof that the smallest $N$
such that $R(s,N)=[0,\frac{s(s-1)}{2}]$ is $N=2s-1$.

\begin{Pro}\label{f+1etc}
Let $2 \leq a \leq s-1$.  For $x=(a-1)+\alpha$ with $0 \leq \alpha
\leq s-a$, let $S_x$ be the $s$-set $[a-1,s+a] \setminus \{x,s+a-1\}
\subseteq [1,s+a]$.\\
Then $\{r(S_x):a-1 \leq x \leq s-1\}$ contains the following values:
\begin{itemize}
\item $\frac{(s-a)(s-a+1)}{2}+\{1+2 \alpha -\delta_{\alpha} \}$, where $0 \leq \alpha \leq \mathrm{min}(a-2,s-a)$;
\item if $\mathrm{min}(a-2,s-a)=a-2$, there are further values of the form
\[\frac{(s-a)(s-a+1)}{2}+\{ a-1 + \alpha -\delta_{\alpha} \} (a-1 \leq \alpha \leq s-a);\]
\end{itemize}
where $\delta_{\alpha}=1$ if $\alpha=\frac{s-a+1}{2}$ and $0$
otherwise.
\end{Pro}
\begin{proof}
Let $I:=[a-1,s+a]$; then $r(I)=\frac{(s-a+3)(s-a+4)}{2}$. Let $a-1
\leq x \leq s-1$; we consider $S_x=I \setminus \{x,s+a-1\}$.

Deleting $\{s+a-1\}$ from $I$: this element cannot occur as a
summand; as a sum, there are $s-a+2$ pairs $(a-1,s),
(a,s-1),\ldots,(s,a-1)=((a-1)+(s-a+1),a-1)$ which sum to $s+a-1$,
and are lost from $r(I)$ upon its deletion.

Deleting $x$ from $I$: since $a-1 \leq x \leq s-1$, $x$ occurs as a
summand and may also occur as a sum.  As a summand, there are
$(s-x+2)$ pairs of the form $(x,a-1),(x,a),\ldots,(x,s+a-x)$;
doubling this to count all pairs will overcount by precisely one, so
there are $2(s-x)+3$ pairs here in total (for: one pair is counted
twice if $x=a-1+j$ for some $0 \leq j \leq s+1-x$; certainly $x \geq
a-1$; also $x \leq s+a-x$ if $x \leq \frac{(2s-1)}{2}<s$). If $a-1
\leq x \leq 2a-3$, $x$ does not occur as a sum; however if $x \geq
2a-2$, there are also pairs corresponding to $x$ as a sum: these
pairs are $(a-1,x-(a-1)), (a,x-a),\ldots,(x-a+1,a-1)$ and so there
are $x-2a+3$ such.  There is clearly no overlap in the sum/summand
counts.

Finally, consider the overlap between pairs counted in the $x$ and
the $(s+a-1)$ cases.  The quantity $s+a-1$ will have $x$ as a
summand in the pairs $(x,s+a-1-x)$ and $(s+a-1-x,x)$; two pairs
unless $x=s+a-1-x$, i.e. $x=\frac{(s+a-1)}{2}$, in which case it is
just a single pair. So after subtracting both quantities for the two
cases, we must add $2$ unless $x=\frac{s+a-1}{2}$, when we add $1$
instead.  No other type of overlap is possible. Hence
\[ r(I \setminus \{x,s+a-1\})=r(I)-(s-a+2)-(2(s-x)+3)-\mathrm{max}(x-(2a-3),0)+(2-\epsilon_x)\]
where $\epsilon_x=1$ if $x=\frac{s+a-1}{2}$ and $0$ otherwise.
Rewriting with $x=(a-1)+\alpha$, where $0 \leq \alpha \leq s-a$, the
right side becomes
$\frac{(s-a+3)(s-a+4)}{2}-(s-a+2)-2(s-a)-(5-2\alpha)-\mathrm{max}(\alpha+2-a,0)+(2-\delta_{\alpha})$
where $\delta_{\alpha}=1$ if $\alpha=\frac{s-a+1}{2}$ and $0$
otherwise. Thus
\[ r(I \setminus \{x,s+a-1\})=(\frac{(s-a)(s-a+1)}{2}+1)+B(s,a,\alpha)-\delta_{\alpha}\]
where $B(s,a,\alpha)=\begin{cases}
2 \alpha,  0 \leq \alpha \leq \mathrm{min}(a-2,s-a);\\
\alpha+a-2,  \mbox{if $a-2<s-a$ and }  a-1 \leq \alpha \leq s-a.\\
\end{cases}$
\end{proof}


The next proposition complements the previous result.

\begin{Pro}\label{fetc}
Let $2 \leq a \leq s-1$.  For $x=(a-1)+\alpha$ with $0 \leq \alpha
\leq s-a$, let $S_x:=[a-1,s+a-1] \setminus \{x\} \subseteq
[1,s+a-1]$.\\
Then $\{r(S_x): a-1 \leq x \leq s-1\}$ contains the following
values:
\begin{itemize}
\item $\frac{(s-a)(s-a+1)}{2}+\{2 \alpha - \delta_{\alpha} \}$, where $0 \leq \alpha \leq \mathrm{min}(a-2,s-a)$;
\item if $\mathrm{min}(a-2,s-a)=a-2$, there are further values of the form
\[\frac{(s-a)(s-a+1)}{2}+\{ a-2 + \alpha - \delta_{\alpha} \} (a-1 \leq \alpha \leq s-a);\]
\end{itemize}
where $\delta_{\alpha}=1$ if $\alpha=\frac{s-a+2}{2}$ and $0$
otherwise.
\end{Pro}
\begin{proof}
The proof of Proposition \ref{f+1etc} can be replicated, with 2
adaptations:
\begin{itemize}
\item in replacing $s+a-1$ by $s+a$ as the deleted element, the number of pairs to be subtracted to account for its deletion is increased by $1$ to $s-a+3$.
\item in considering the overlap between pairs counted in the $x$ and the $(s+a)$ cases, $s+a$ will have $x$ as a summand in the pairs $(x,s+a-x)$ and $(s+a-x,x)$;
these pairs are distinct unless $x=s+a-x$, i.e. $x=\frac{(s+a)}{2}$.
Hence, after subtracting both quantities for the two cases, $2$ must
be added unless $x=\frac{s+a}{2}$, when $1$ must be added instead.
\end{itemize}
\end{proof}

Note that in Proposition \ref{f+1etc}, the $s$-sets attaining
$r$-vales close to $f_{s,s+a-1}$ lie in $[1,s+a-1]$, whereas in
Proposition \ref{fetc} the $s$-sets lie in $[1,s+a]$ but not in
$[1,s+a-1]$.

\begin{Pro}\label{initialsegment}
For $2 \leq a \leq s-1$, $R(s,s+a)$ contains
$[\frac{(s-a)(s-a+1)}{2},\frac{(s-a+1)(s-a+2)}{2}-1]$.
\end{Pro}
\begin{proof}
To establish that each stated value can be attained as the $r$-value
of an $s$-set in $[1,s+a]$, combine Propositions \ref{f+1etc} and
\ref{fetc}.  The first proposition yields alternate values starting
at $\frac{(s-a)(s-a+1)}{2}+1$ up to
$\frac{(s-a)(s-a+1)}{2}+1+2\mathrm{min}(a-2,s-a)$; and if
$\mathrm{min}(a-2,s-a)=a-2$, all subsequent values up to
$\frac{(s-a)(s-a+1)}{2}+(s-1)$.  The second proposition yields
alternate values starting at $\frac{(s-a)(s-a+1)}{2}$ up to
$\frac{(s-a)(s-a+1)}{2}+2\mathrm{min}(a-2,s-a)$; and if
$\mathrm{min}(a-2,s-a)=a-2$, all subsequent values up to
$\frac{(s-a)(s-a+1)}{2}+(s-2)$.  We now need only consider the
exceptional cases $x=\frac{s+a-1}{2}=(a-1)+\frac{(s-a+1)}{2}$ or
$x=\frac{s+a}{2}=(a-1)+\frac{(s-a+2)}{2}$, when the ``expected"
$r$-value does not occur. Since $s+a$ is either even or odd for a
given pair $(s,a)$, only one of these exceptions can occur for a
given pair $(s,a)$.  What effect does this have on the spectrum of
attained values?  If the ``missed" value corresponds to the first
type of $r$-value, this $r$-value is either
$\frac{(s-a)(s-a+1)}{2}+(s-a+1)$ (if $s+a$ odd) or
$\frac{(s-a)(s-a+1)}{2}+(s-a+2)$ (if $s+a$ even). In both cases,
these exceed the stated range and so are not required.  If the
exceptional $x$ corresponds to the second type of $r$-value, i.e.
$a-2<s-a$ and the added quantity is $\{ a-2 + \alpha\}$ where $a-1
\leq \alpha \leq s-a$, then in all cases the missed value will be
attained by the construction in the other proposition.  (This is
immediate except for smallest and largest values; the least obvious
case is if $s+a$ is odd and the ``missed" $r$-value is
$\frac{(s-a)(s-a+1)}{2}+(s-1)$; however this can only occur if
$\alpha=s-a=\frac{(s-a+1)}{2}$, i.e $a=2$ and $s=3$, and it can be
easily checked that the required values are obtained.)
\end{proof}

We are now ready to prove the final result of this section.

\begin{Thm}\label{eachsegment}
For $N \geq s+1$, $R(s,N)$ contains $[f_{s,N-1},g_{s,N-1}]$.
\end{Thm}
\begin{proof}
We prove that, for $a \in \mathbb{N}$, $R(s,s+a)$ contains all of
$[f_{s,s+a-1},g_{s,s+a-1}]$. We induct on $a \in \mathbb{N}$. The
base case is $a=1$; it is clear that $R(s,s+1)$ contains
$[f_{s,s},g_{s,s}]=\{\frac{s(s-1)}{2}\}$ (take set $[1,s]$). Now let
$a=A>1$ and suppose the result holds for $a=A-1$. If $A-1 \geq s$,
then $s-A+1 \leq 0$ and so $R[s,s+A-1]$ contains
$[0,\frac{s(s-1)}{2}]$, hence so does $R(s,s+A)$. If $A-1=s-1$, i.e.
$A=s$, then we must show $R(s,2s)$ contains $[0,\frac{s(s-1)}{2}]$
given that $R(s,2s-1)$ contains $[1,\frac{s(s-1)}{2}]$; this is
easily seen since $r([s+1,2s])=0$. So we may assume $2 \leq A \leq
s-1$. Now, $R(s,s+A-1)$ contains
$[\frac{(s-A+1)(s-A+2)}{2},\frac{s(s-1)}{2}]$ by induction, while
applying Proposition \ref{initialsegment}  guarantees the occurrence
of all $r$-values in the range
$[\frac{(s-A)(s-A+1)}{2},\frac{(s-A+1)(s-A+2)}{2}-1]$. Hence
$R(s,s+A)$ contains $[\frac{(s-A)(s-A+1)}{2},\frac{s(s-1)}{2}]$.
\end{proof}

\subsection{Describing the exceptional values}

We are aiming to show that $R(s,N)$ equals $[f_{s,N},g_{s,N}]$ with
some ``missing" values if and only if $s+2 \leq N \leq 2s-2$. The
following theorem shows that, for $N$ in the stated range, it is
never possible for the size-$s$ subsets of $[1,N]$ to attain all
values in the interval $[f_{s,N},g_{s,N}]$.

\begin{Thm}\label{nof+1}
For $s+2 \leq N \leq 2s-2$, $R(s,N)$ does not contain $f_{s,N}+1$.
\end{Thm}
\begin{proof}
Let $1 \neq a \in \mathbb{N}$.  We will prove that, for $s \geq
a+2$, $R(s,s+a)$ does not contain $f_{s,s+a}+1$.  We will use
induction on $s$. The base case is $s=a+2$: we must show that
$R(a+2,2a+2)$ does not contain $f_{a+2,2a+2}+1$, i.e. that
$R(b,2b-2)$ does not contain $f_{b,2b-2}+1$ for $b=a+2(\geq 4) \in
\mathbb{N}$.  Here $N=2b-2 \geq 6$, $b=\frac{N}{2}+1$ and
$f_{b,N}+1=2$; this is precisely the result proved in Theorem
\ref{risnot2}.

Now let $m>a+2$ and suppose the result holds for $m-1$,
i.e.$R(m,m-1+a)$ does not contain $f_{m,m-1+a}+1$. We will show that
$R(m,m+a)$ does not contain $f_{m,m+a}+1$.

Consider the size-$m$ subsets of $[1,m+a]$.  Those which lie in
$[1,m+a-1]$ have $r$-values in the range $f_{m,m+a-1} \leq r \leq
g_{m,m+a-1}$, i.e $\frac{(m-a)(m-a+1)}{2} \leq r \leq
\frac{m(m-1)}{2}$.  For the interval $[1,m+a]$, the range of
$r$-values of $m$-sets is $\frac{(m-a-1)(m-a)}{2} \leq r \leq
\frac{s(s-1)}{2}$.  So, those $m$-sets with $r$-values in the range
$\frac{(m-a-1)(m-a)}{2} \leq r \leq \frac{(m-a)(m-a+1)}{2}-1$ must
contain the maximum element $m+a$. Let $S$ be such a set; then $S=T
\cup \{m+a\}$ where $T$ is an $(m-1)$-set contained in $[1,m+a-1]$.
The range of possible $r$-values for $T$ is
$\frac{(m-a-2)(m-a-1)}{2} \leq r \leq \frac{(m-1)(m-2)}{2}$.

Consider $r(T\cup \{m+a\})$.  Since all elements of $T$ are smaller
than $m+a$, clearly $t+(m+a) \not\in T$ for all $t \in T$; also
$2(m+a) \not \in T \cup \{m+a\}$.  So any new contribution to the
$r$-value from the adjoining of $m+a$ must correspond to its arising
as a sum in $T+T$.  Now, in $[1,m+a-1]^2$ there are $m+a-1$ pairs
which sum to $m+a$, i.e. $(1,m+a-1),(2,m+a-2), \ldots (m+a-1,1)$.
Since $T$ has size $m-1$, $a$ points have been deleted from
$[1,m+a-1]$ to obtain $T$, say $\{ x_1,\ldots,x_a \}$.  How many of
the $(m+a-1)$ pairs have been lost?  The minimum possible is $a$ (if
the $x_i$ form $\frac{a}{2}$ pairs which each sum to $m+a$) and the
maximum is $2a$ (if $x_i+x_j \neq m+a$ for all $1 \leq i,j \leq a$;
possible since $a<m$ and hence $a< \frac{m+a}{2}$).  So $r(T \cup
\{m+a\})= r(T)+(m+a-1)-(a+\alpha)=r(T)+(m-1-\alpha)$ where $0 \leq
\alpha \leq a$.

The smallest possible value obtainable for $r(T \cup \{m+a\})$ from
this formula would correspond to taking minimum possible
$r(T)=f_{m-1,m+a-1}=\frac{(m-a-2)(m-a-1)}{2}$ and maximum possible
$\alpha=a$.  This is valid, since there is a unique $T$ with
$r(T)=f_{m-1,m+a-1}$, namely $T=[a+1,m+a-1]$, and $\alpha=a$ in this
case.  This yields $r(T \cup
\{m+a\})=f_{m,m+a}=\frac{(m-a-1)(m-a)}{2}$.

Now consider how an $r$-value of $f_{m,m+a}+1$ can be obtained.  By
previous discussion, any $S \subseteq [1,m+a]$ with $r(S)=f_{m,m+a}$
must be of the form $S=T \cup \{m+a\}$; we consider the possible $T$
and $\alpha$. Suppose $r(T \cup \{m+a\})=\frac{(m-a-1)(m-a)}{2}+1$.
In terms of the above formula, the possibilities are:
\begin{itemize}
\item $r(T)=\frac{(m-a-2)(m-a-1)}{2}$, $m-1-\alpha=m-a$;
\item $r(T)=\frac{(m-a-2)(m-a-1)}{2}-i$, $m-1-\alpha=m-a+i$ for some $i \in \mathbb{N}$;
\item $r(T)=\frac{(m-a-2)(m-a-1)}{2}+i$, $m-1-\alpha=m-a-i$ for some $i \in \mathbb{N}$.
\end{itemize}
But clearly the first option is impossible, as there is a unique $T$
with this $r$-value, which has $\alpha=a$ and was dealt with above.
The second case is also impossible, since it requires $r(T)$ to be
less than the minimum possible.  The only possibility is the third
case, but here we must have $i=1$; we cannot have $i>1$ since this
would force $\alpha=a+i-1>a$.  Thus, there exists an $m$-set
$S=T\cup \{m+a\}\subseteq [1,m+a]$ with $r(S)=f_{m,m+a}+1$ only if
there exists an $(m-1)$-set $T \subseteq [1,(m-1)+a]$ with
$r(T)=f_{m-1,(m-1)+a}+1$.  But by the Induction Hypothesis,
$R(m-1,m-1+a)$ does not contain $f_{m-1,m-1+a}+1$, i.e. no such $T$
exists; hence no such $S$ exists and $R(m,m+a)$ does not contain
$f_{m,m+a}+1$.
\end{proof}

We now establish that all exceptional $r$-values must be of the form
$f_{s,N}+k$ where $k$ is odd and lies in a restricted range.

\begin{Pro}\label{smallf+odd}
Let $1 \leq a \leq s-2$. For $x=(a+1)-\alpha$ where $0 \leq \alpha
\leq a$, let $S_x:=\{x\} \cup [a+2,s+a] \subseteq [1,s+a]$.  Then
$\{r(S_x): 1 \leq x \leq a+1\}$ contains the following values:
\begin{itemize}
\item $\frac{(s-a-1)(s-a)}{2}+\{2 \alpha\}$, where $0 \leq \alpha \leq \frac{a}{2}$;
\item $\frac{(s-a-1)(s-a)}{2}+\{ 2\alpha-1\}$ where $\frac{a}{2} < \alpha \leq a$.
\end{itemize}
\end{Pro}
\begin{proof}
Let $I=[a+2,s+a]$. By Lemma \ref{Int},
$r(I)=\frac{(s-a-3)(s-a-2)}{2}$ for $1 \leq a+2 \leq s-1$ and $0$
for $a+2>s-1$.  We first consider the former case, i.e. $1 \leq a
\leq s-3$.  Consider adjoining an element $x$, $1 \leq x \leq a+1$.
The element $x$ cannot occur as a sum with summands from $I \cup
\{x\}$; we consider its role as a summand.  The pair $(x,x)$ yields
$2x \in I$ precisely if $a+2 \leq 2x \leq s+a$.  Since $x \leq a+1
\leq a+\frac{3}{2} \leq \frac{2a+3}{2} \leq \frac{s+a}{2}$ (using $a
\leq s-3$), $2x$ can never be too large to lie in $I$, and so $2x
\in I$ precisely if $\frac{a+2}{2} \leq x \leq a+1$ (clearly
$\frac{a+2}{2} \leq a+1$ for all possible $a$).  Finally, there are
$s-x-1$ pairs $(x,a+2),(x,a+3),\ldots,(x,s+a-x)$ which sum to $x$,
and doubling these yields the total (here there is no over-count
since $x$ does not lie in $[a+2,s+a]$). So, $r(I \cup
x)=r(I)+2(s-x-1)$ if $ 1 \leq x < \frac{a+2}{2}$ and $r(I \cup
x)=r(I)+2(s-x-1)+1$ if $\frac{a+2}{2} \leq x \leq a+1$. Writing
$x=(a+1)-\alpha$ where $0 \leq \alpha \leq a$ yields
\[ r(I \cup x)=\frac{(s-a-1)(s-a)}{2}+\begin{cases} 2\alpha, \quad 0 \leq \alpha \leq \frac{a}{2}\\ 2\alpha-1, \quad \frac{a}{2} < \alpha \leq a. \end{cases}\]
For the case $a=s-2$, $I=[s,2s-2]$ and $r(I)=0$; we adjoin $1 \leq x
\leq s-1$.  The element $x$ does not occur as a sum with summands
from $I \cup \{x\}$.  As a summand, there is a contribution of $1$
to the count from the pair $(x,x)$ precisely if $\frac{s}{2} \leq x
\leq s-1$, and a contribution of $2(s-x-1)$  from pairs
$(x,s),\ldots,(x,2s-2-x)$ and $(s,x),\ldots,(2s-2-x,x)$ (no
duplication since $x \not \in I$).  Then $r(I \cup x)=2(s-x-1)$ if
$1 \leq x < \frac{s}{2}$ and $r(I \cup x)=2(s-x-1)+1$ if
$\frac{s}{2} \leq x \leq s-1$, i.e. writing $x=(s-1)-\alpha$, $0
\leq \alpha \leq s-2$, \[r(I \cup x)= \begin{cases} 2 \alpha+1,
\quad 0 \leq \alpha \leq \frac{s-2}{2},\\ 2 \alpha, \quad
\frac{s-2}{2} < \alpha \leq s-2. \end{cases}\] Since here
$\frac{(s-a-1)(s-a)}{2}=1$, this is of the same form as the main
case.
\end{proof}

Combining several previous results yields the following description and upper bound for the set of exceptions.

\begin{Thm}\label{exceptionalA}
Suppose $s+2 \leq N \leq 2s-2$. Then for any $r \in
[f_{s,N},g_{s,N}]$ such that $r \not\in R(s,N)$,
\[ r \in \{f_{s,N}+(2i-1):1 \leq i \leq \mathrm{min}(s-\lceil \frac{N}{2}
\rceil, \lfloor \frac{N-s}{2} \rfloor) \}.\]
\end{Thm}
\begin{proof}
We prove that for $2 \leq a \leq s-2$, $R(s,s+a)$ contains
\[ [f_{s,s+a},g_{s,s+a}] \setminus \{f_{s,s+a}+(2i-1):1 \leq i \leq \mathrm{min}(\lfloor \frac{a}{2} \rfloor, \lfloor \frac{s-a}{2} \rfloor) \}.\]
We apply a sequence of results to obtain as many $r$-values as
possible in $[f_{s,s+a},g_{s,s+a}]$; at each stage we describe the
set of ``missing" values, i.e. values in $[f_{s,s+a},g_{s,s+a}]$
which are not obtainable by these methods.  Recall that
$f_{s,s+a}=\frac{(s-a-1)(s-a)}{2}$.
\begin{itemize}
\item Proposition \ref{initialsegment} guarantees that $R(s,s+a)$ contains
the interval
$[f_{s,s+a-1},g_{s,s+a-1}]=[\frac{(s-a)(s-a+1)}{2},\frac{s(s-1)}{2}]$;
this shows that any missing values must lie in $\{f_{s,s+a}+i:
i=0,1,2,\ldots,s-a-1\}$.
\item Proposition \ref{fetc} (applied with $a^{\prime}=a+1$) yields the following as $r$-values of $s$-sets in $[1,s+a^{\prime}-1]=[1,s+a]$:
\begin{itemize}
\item $\frac{(s-a-1)(s-a)}{2}+\{2 \alpha\}, \quad 0 \leq \alpha \leq \mathrm{min}(a-1,s-a-1)$,
\item and if $a-1<s-a-1$, $\frac{(s-a-1)(s-a)}{2}+\{a-1+\alpha\}, \quad a \leq \alpha \leq s-a-1$.
\end{itemize}
with the exception of $\alpha=\frac{s-a+1}{2}$ which lies outside
our range. Hence if $a-1 < s-a-1$, we obtain
$f_{s,s+a}+\{0,2,\ldots,2a-2\}$ followed by $f_{s,s+a}+[2a-1,s-2]$.
If $s-a-1 \leq a-1$ we obtain $f_{s,s+a}+\{0,2,\ldots,2(s-a-1)\}$;
certainly $2(s-a-1) \geq s-a-1$.  This shows that any missing values
must lie in $\{f_{s,s+a}+(2i-1):
i=1,2,\ldots,\mathrm{min}(a-1,\lfloor \frac{s-a}{2} \rfloor)\}$.
\item Proposition \ref{smallf+odd} guarantees the existence of:
\begin{itemize}
\item $\frac{(s-a-1)(s-a)}{2}+\{2 \alpha\}$, where $0 \leq \alpha \leq \frac{a}{2}$;
\item $\frac{(s-a-1)(s-a)}{2}+\{ 2\alpha-1\}$ where $\frac{a}{2} < \alpha \leq a$.
\end{itemize}
Hence we obtain $f_{s,s+a}+\{a,a+2,\ldots, 2a-1 \}$ ($a$ odd) or
$f_{s,s+a}+\{a+1,a+3,\ldots, 2a-1 \}$ ($a$ even).  This shows that
the missing values must lie in $\{f_{s,s+a}+(2i-1):
i=1,2,\ldots,\mathrm{min}(\lfloor \frac{a}{2} \rfloor, a-1,\lfloor
\frac{s-a}{2} \rfloor)\}$=$\{f_{s,s+a}+(2i-1):
i=1,2,\ldots,\mathrm{min}(\lfloor \frac{a}{2} \rfloor,\lfloor
\frac{s-a}{2} \rfloor)\}$.

Finally, set $N=s+a$ to see that $\mathrm{min}(\lfloor \frac{a}{2}
\rfloor,\lfloor \frac{s-a}{2} \rfloor)\}=\mathrm{min}(s-\lceil
\frac{N}{2} \rceil, \lfloor \frac{N-s}{2} \rfloor)$.
\end{itemize}
\end{proof}

We now combine the results of the previous sections to establish our
main theorem.

\begin{Thm}\label{main}
Let $1 \neq s \in \mathbb{N}$.  Denote by $R(s,N):=\{r(S): S \subseteq [1,N], |S|=s\}$, let $f_{s,N}$ be the smallest element in $R(s,N)$ and let $g_{s,N}$ be the largest element in $R(s,N)$.  Then
\begin{itemize}
\item[(i)] for $N=s$ and $N=s+1$,
$$R(s,N)=[f_{s,N},g_{s,N}]=[\frac{(2s-N)(2s-N-1)}{2},\frac{s(s-1)}{2}]$$ where $f_{s,N}>0$;

\item[(ii)] for  $s+2 \leq N \leq 2s-2$, $$R(s,N)=[f_{s,N},g_{s,N}] \setminus \{x_1,\ldots,x_e\}$$ where
$f_{s,N}=\frac{(2s-N)(2s-N-1)}{2}>0$, $g_{s,N}=\frac{s(s-1)}{2}$,
and
   $$\{f_{s,N}+1 \} \subseteq \{x_1,\ldots,x_e\} \subseteq \{f_{s,N}+(2i-1):1 \leq i \leq \mathrm{min}(s-\lceil \frac{N}{2}
 \rceil, \lfloor \frac{N-s}{2} \rfloor) \};$$
\item[(iii)] for $N \geq 2s-1$, $$R(s,N)=[f_{s,N},g_{s,N}]=[0,\frac{s(s-1)}{2}].$$
\end{itemize}
\end{Thm}
\begin{proof}
In all cases, $f_{s,N}$ and $g_{s,N}$ are given by Theorems \ref{fs}
and \ref{gs}. The absence of exceptions in the $N=s$ case is trivial
and in the $N=s+1$ case follows from Proposition \ref{N=s+1}.
Theorem \ref{nof+1} establishes that $f_{s,N}+1$ is an exception for
each $s+2 \leq N \leq 2s-2$, while Theorem \ref{exceptionalA} shows
that the stated values are attained.  For the third part, we apply
Theorem \ref{eachsegment} with $N=2s-1$. This establishes that
$s$-sets in $[1,2s-1]$ attain all $r$-values in
$[1,\frac{s(s-1)}{2}]$. The remaining $r$-value $0$ is obtained by
the $s$-set $[s,2s-1] \subseteq [1,2s-1]$.
\end{proof}

\section{Concluding remarks}

In this paper, a comprehensive description has been given of the
behaviour of the $r$-values of subsets of $[1,N] \subseteq
\mathbb{N}$.  The range of possible $r$-values for any $s$-set has
been described, and exceptional values have been shown to possess a
very specific form.  We end this paper by conjecturing that the set
shown to contain the exceptional values is precisely the set of
exceptions:

\begin{Conj}\label{conj}
Let $1 \neq s \in \mathbb{N}$.  Denote by $R(s,N):=\{r(S): S \subseteq [1,N], |S|=s\}$, let $f_{s,N}$ be the smallest element in $R(s,N)$ and let $g_{s,N}$ be the largest element in $R(s,N)$.  Then for  $s+2 \leq N \leq 2s-2$,

$$R(s,N)=[f_{s,N},g_{s,N}] \setminus \{f_{s,N}+(2i-1):1 \leq i \leq \mathrm{min}(s-\lceil \frac{N}{2}
 \rceil, \lfloor \frac{N-s}{2} \rfloor) \}$$

where $f_{s,N}=\frac{(2s-N)(2s-N-1)}{2}>0$ and
$g_{s,N}=\frac{s(s-1)}{2}$.
\end{Conj}

Inductive proof strategies for this conjecture, extending the
approach used for Theorem \ref{nof+1}, encounter problems due to the
presence of the floor and ceiling functions, suggesting that an
alternative approach may be required.

\end{document}